\documentclass[oneside,english,reqno]{amsart}
\usepackage[T1]{fontenc}
\usepackage[latin9]{inputenc}
\usepackage{babel}
\usepackage{amsthm}
\usepackage{amssymb}
\usepackage{graphicx}
\usepackage[unicode=true,pdfusetitle,
 bookmarks=true,bookmarksnumbered=false,bookmarksopen=false,
 breaklinks=false,pdfborder={0 0 1},backref=false,colorlinks=false]
 {hyperref}
\usepackage{breakurl}

\makeatletter
\numberwithin{equation}{section}
\numberwithin{figure}{section}
\newenvironment{lyxlist}[1]
{\begin{list}{}
{\settowidth{\labelwidth}{#1}
 \setlength{\leftmargin}{\labelwidth}
 \addtolength{\leftmargin}{\labelsep}
 }}
{\end{list}}
\theoremstyle{plain}
\newtheorem{thm}{\protect\theoremname}[section]
  \theoremstyle{definition}
  \newtheorem{defn}[thm]{\protect\definitionname}
  \theoremstyle{plain}
  \newtheorem{prop}[thm]{\protect\propositionname}
  \theoremstyle{plain}
  \newtheorem{algorithm}[thm]{\protect\algorithmname}
  \theoremstyle{remark}
  \newtheorem{acknowledgement}[thm]{\protect\acknowledgementname}

\newcommand{\gph}{\mbox{\rm Graph}}

\makeatother

  \providecommand{\acknowledgementname}{Acknowledgement}
  \providecommand{\algorithmname}{Algorithm}
  \providecommand{\definitionname}{Definition}
  \providecommand{\propositionname}{Proposition}
\providecommand{\theoremname}{Theorem}

\begin{document}
\title[Nonconvex inequality problems]{Finitely convergent algorithm for nonconvex inequality problems}

\begin{abstract}
We extend Fukushima's result on the finite convergence of an algorithm
for the global convex feasibility problem to the local nonconvex case.
\end{abstract}

\author{C.H. Jeffrey Pang}

\curraddr{Department of Mathematics\\ 
National University of Singapore\\ 
Block S17 08-11\\ 
10 Lower Kent Ridge Road\\ 
Singapore 119076 }

\email{matpchj@nus.edu.sg}


\date{\today{}}

\maketitle
\tableofcontents{}

\section{Introduction}

Let $X$ be a Hilbert space. We consider the \emph{Nonconvex Inequality
Problem }(NIP) 
\begin{equation}
\mbox{(NIP):}\quad\mbox{ For }f:X\to\mathbb{R}\mbox{, find }x\in\mathbb{R}^{n}\mbox{ s.t. }f(x)\leq0.\label{eq:CIP}
\end{equation}
In \cite{Fukushima82}, Fukushima proposed a simple \emph{global}
algorithm when $X=\mathbb{R}^{n}$ for the Convex Inequality Problem
(which is the NIP with the additional requirement that $f(\cdot)$
is convex) that converges to some point $\bar{x}$ such that $f(\bar{x})\leq0$
if the Slater condition (i.e., the existence of a point $x^{*}$ satisfying
$f(x^{*})<0$) is satisfied. The ideas can be easily extended to the
case when $X$ is a Hilbert space, and the function $f(\cdot)$ need
not be smooth. In this paper, we make use of tools in nonsmooth and
variational analysis \cite{Cla83,Mor06,RW98} to prove a local result
on the case where $f(\cdot)$ is nonconvex. 

We now discuss some problems related to the NIP. In the case where
$f(\cdot)$ can be written as a maximum of finitely many smooth functions,
a variant of the Newton method converges superlinearly, and global
convergence is possible when $f(\cdot)$ is the maximum of finitely
many smooth convex functions. We refer to the references stated in
\cite{Fukushima82} for more details. (It appears that \cite{DePierroIusem88}
have obtained similar results independently.)

In \cite{Robinson_76_CIP}, Robinson considered the \emph{$K$-Convex
Inequality Problem} (KCIP), which is a generalization of the (CIP).
For $f:\mathbb{R}^{n}\to\mathbb{R}^{m}$, and a closed convex cone
$K\subset\mathbb{R}^{m}$, we write $y_{1}\leq_{K}y_{2}$ if $y_{2}-y_{1}\in K$.
The KCIP is defined by 
\begin{equation}
\mbox{(KCIP):}\mbox{ For }f:\mathbb{R}^{n}\to\mathbb{R}^{m}\mbox{ and }C\subset\mathbb{R}^{n}\mbox{, find }x\in C\mbox{ s.t. }f(x)\leq_{K}0.\label{eq:KCIP}
\end{equation}
Robinson's algorithm in \cite{Robinson_76_CIP} for the CIP can be
described as follows: At each iterate $x_{i}$, a subgradient $y_{i}\in\partial f(x_{i})$
is obtained, and the halfspace 
\begin{equation}
H_{i}^{\leq}:=\{x\in\mathbb{R}^{n}\mid f(x_{i})+\langle y_{i},x-x_{i}\rangle\leq0\}\label{eq:halfspace-type}
\end{equation}
 contains $f^{-1}((-\infty,0])$. The next iterate $x_{i+1}$ is obtained
by projecting $x_{i}$ onto $H_{i}^{\leq}$. Assuming regularity and
convexity (and no smoothness), Robinson proved that the algorithm
for the KCIP converges at least linearly. With smoothness, superlinear
convergence can be expected.

Modifications for a finitely convergent algorithm for the NIP can
be traced back to \cite{Polak_Mayne79,Mayne_Polak_Heunis81}, where
$f(\cdot)$ is a maximum of finitely many smooth functions. The main
idea for obtaining finite convergence under the Slater condition can
be described as follows. Instead of trying to find $x$ such that
$f(x)\leq0$, an infinite sequence $\{\epsilon_{i}\}$ of positive
numbers is introduced, and one tries to find $x$ satisfying
\[
f(x)\leq-\epsilon_{i}
\]
 in the $i$th iteration. The contribution in \cite{Fukushima82}
is to show that the smoothness conditions can be dropped. For more
recent work, we refer the reader to \cite{BauschkeWangWangXu14,CensorChenPajoohesh11,Crombez04}
and the references therein.

A problem related to the NIP is the Set Intersection Problem (SIP).
For sets $K_{1},\dots,K_{r}$ in a Hilbert space $X$, the SIP is
stated as: 
\begin{equation}
\mbox{(SIP):}\quad\mbox{Find }x\in K:=\bigcap_{i=1}^{r}K_{i}\mbox{, where }K\neq\emptyset.\label{eq:SIP}
\end{equation}
The SIP can be seen as a particular case of the NIP: Take $f(\cdot)$
to be $\max_{i=1,\dots,r}d(x,K_{i})$. A common method of solving
such problems is the method of alternating projections, which typically
has linear convergence even in convex problems. There has been recent
interest in nonconvex problems \cite{LewisMalick08,LLM09_lin_conv_alt_proj},
where the research is focused on conditions for the linear convergence
of the method of alternating projections and its variants.

We also remark that the NIP is related to filter methods for nonlinear
programming \cite{FletcherLeyffer02MAPR}.

\subsection{Contributions of this paper}

In this paper, we prove a local result on the finite convergence of
an algorithm for the NIP \eqref{eq:CIP} when $f(\cdot)$ is approximately
convex \cite{NgaiLucThera00,DaniilidisGeorgiev04} (See Definition
\ref{def:approx-convex} and the subsequent commentary) and $X$ is
a Hilbert space.

\subsection{Notation}

Let $X$ be a Hilbert space, and let $x\in X$ and $S\subset X$.
The following notation we will use is quite standard. 
\begin{lyxlist}{00.00.0000}
\item [{$\mathbb{B}(x,r)$}] The closed ball with center $x$ and radius
$r$.
\item [{$d(x,S)$}] The distance from $x$ to $S$. 
\end{lyxlist}

\section{Preliminaries}

In this section, we provide the necessary background in variational
analysis for the proof of our algorithm for the NIP. We first recall
the Clarke subdifferential.
\begin{defn}
\label{def:Clarke-subdif}(Clarke subdifferential) Let $X$ be a Hilbert
space. Consider a function $f:X\rightarrow\mathbb{R}$ locally Lipschitz
at a point $\bar{x}\in X$. The \emph{Clarke (generalized) subdifferential}
of $f$ at $\bar{x}\in X$ is defined by 
\begin{equation}
\partial f(\bar{x}):=\{x^{*}\in X^{*}:\langle x^{*},d\rangle\leq f^{0}(\bar{x};d)\mbox{ for all }d\in X\},\label{eq:Clarke-subdif}
\end{equation}
where $f^{0}(\cdot;\cdot)$ is the \emph{Clarke (generalized) directional
derivative }defined by 
\begin{equation}
f^{0}(\bar{x};d):=\limsup_{(y,t)\to(x_{0},0^{+})}\frac{f(y+td)-f(y)}{t}.\label{eq:Clarke-dir-diff}
\end{equation}

\end{defn}
We now describe the nonconvex functions for which we are able to prove
finite convergence of our algorithm. 
\begin{defn}
\cite{NgaiLucThera00}\label{def:approx-convex} (Approximate convexity)
Let $X$ be a Hilbert space and $f:X\to\mathbb{R}$. We say that $f(\cdot)$
is \emph{approximately convex at $\bar{x}$} if for every $\epsilon>0$,
there exists $\delta>0$ such that 
\[
f(y)\geq f(x)+\langle s,y-x\rangle-\epsilon\|y-x\|\mbox{ for all }x,y\in\mathbb{B}(\bar{x},\delta)\mbox{ and }s\in\partial f(x).
\]

\end{defn}
The notion of weak convexity in \cite{Vial83} (see also \cite{HiriartUrruty1984}
and the references therein) was a precursor to the notion of approximate
convexity in \cite{NgaiLucThera00}. The definition of approximate
convexity above is different from its usual definition, but is equivalent
by \cite[Theorem 1]{DaniilidisGeorgiev04}. In the case where $X=\mathbb{R}^{n}$,
approximate convexity is equivalent to $f(\cdot)$ being lower-$\mathcal{C}^{1}$
\cite{DaniilidisGeorgiev04,Spingarn81,AusselDaniilidisThibault04}.
Lower-$\mathcal{C}^{1}$ functions include the pointwise maximum of
a finite number of $\mathcal{C}^{1}$ functions. We refer to \cite[Section 10G]{RW98}
and the references therein for a discussion on lower-$\mathcal{C}^{1}$
functions, and more generally, subsmooth functions.

We now recall metric regularity.
\begin{defn}
(Metric regularity) Let $S:X\rightrightarrows Y$ be a set-valued
map. We say that $S(\cdot)$ is metrically regular at $(\bar{x},\bar{y})$
if there exist a constant $\kappa\geq0$ and neighborhoods $U$ of
$\bar{x}$ and $V$ of $\bar{y}$ such that 
\[
d\big(x,S^{-1}(u)\big)\leq\kappa d\big(u,S(x)\big)\mbox{ for all }x\in U\mbox{ and }u\in W.
\]

\end{defn}
We now make a claim about locally Lipschitz functions.
\begin{prop}
\label{prop:Metric-reg-epigph}(Metric regularity of epigraphical
maps) Let $X$ be a Hilbert space, and let $f:X\to\mathbb{R}$ be
locally Lipschitz at $\bar{x}\in X$. If $0\notin\partial f(\bar{x})$,
then the epigraphical map $E:X\rightrightarrows\mathbb{R}$ defined
by $E(x):=[f(x),\infty)$ is metrically regular at $(\bar{x},f(\bar{x}))$.\end{prop}
\begin{proof}
We make use of the Aubin criterion in \cite[Theorem 1.2]{DQZ06},
but we need to recall a few definitions. Let $X$ and $Y$ be Banach
spaces. For a positively homogeneous map $H:X\rightrightarrows Y$,
the\emph{ inner norm} $\|\cdot\|^{-}$ is defined as 
\[
\|H\|^{-}:=\sup_{x\in\mathbb{B}}\inf_{y\in H(x)}\|y\|.
\]
For a set-valued map $S:X\rightrightarrows Y$, consider $(\bar{x},\bar{y})$
such that $\bar{y}\in S(\bar{x})$, or $(\bar{x},\bar{y})\in\gph(S)$.
The\emph{ graphical (contingent) derivative} of $S$ at $(\bar{x},\bar{y})$
is defined by 
\[
\gph(DS(\bar{x}\mid\bar{y})):=T_{\scriptsize\gph(S)}(\bar{x},\bar{y}),
\]
where the tangent cone $T_{\scriptsize\gph(S)}(\bar{x},\bar{y})$
is defined as follows: $(u,v)\in T_{\scriptsize\gph(S)}(\bar{x},\bar{y})$
if and only if there exists sequences $t_{n}\searrow0$, $u_{n}\to u$
and $v_{n}\to v$ such that $\bar{y}+t_{n}v_{n}\in S(\bar{x}+t_{n}u_{n})$.

The Aubin criterion states that for Banach spaces $X$ and $Y$ and
a set-valued map $S:X\rightrightarrows Y$, $S(\cdot)$ is metrically
regular at $(\bar{x},\bar{y})$ if 
\begin{equation}
\limsup_{{(x,y)\to(\bar{x},\bar{y})\atop (x,y)\in\scriptsize\gph(S)}}\|DS(x\mid y)^{-1}\|^{-}\label{eq:aubin-crit-formula}
\end{equation}
is finite.

We now apply the Aubin criterion to our particular setting. Since
$0\notin\partial f(\bar{x})$, by the formulation of the Clarke subdifferential
using the Clarke directional derivative \eqref{eq:Clarke-dir-diff},
there exists a direction $d$, where $\|d\|=1$, such that $f^{0}(\bar{x},d)<-\mu$,
where $\mu>0$. This means that if $(x,t)$ are close enough to $(\bar{x},0^{+})$,
then 
\[
\frac{f(x+td)-f(x)}{t}<-\mu.
\]
This in turn implies that $(d,-\mu)\in\gph\big(DE\big(x\mid f(x)\big)\big)$.
In other words, 
\begin{equation}
\frac{1}{\mu}d\in DE\big(x\mid f(x)\big)^{-1}(-1).\label{eq:DE1}
\end{equation}
Since $(0,1)$ is a recession direction in $\gph(S)$, it is clear
that 
\begin{equation}
0\in DE\big(x\mid f(x)\big)^{-1}(1).\label{eq:DE2}
\end{equation}
Whenever $y>f(x)$ and $x$ is close enough to $\bar{x}$, the local
Lipschitz continuity of $f(\cdot)$ at $\bar{x}$ ensures that $(x,y)$
is in the interior of the epigraph of $f$, from which we get 
\begin{equation}
0\in DE(x\mid y)^{-1}(1)\mbox{ and }0\in DE(x\mid y)^{-1}(-1)\mbox{ whenever }y>f(x).\label{eq:DE3}
\end{equation}
The formulas \eqref{eq:DE1}, \eqref{eq:DE2} and \eqref{eq:DE3}
combine to give us $\|DE(x\mid y)^{-1}\|^{-}\leq1/\mu$ for all $(x,y)\in\gph(E)$
close enough to $(\bar{x},f(\bar{x}))$. Thus the Aubin criterion
applies to give us the metric regularity of $E(\cdot)$ at $(\bar{x},f(\bar{x}))$. 
\end{proof}

\section{Finitely convergent algorithm for the NIP}

We now present our algorithm for the NIP, and prove its finite convergence
in Theorem \ref{thm:Fin-conv-NIP}.
\begin{algorithm}
\label{alg:fin-conv-alg}(Finitely convergent algorithm for NIP) Let
$X$ be a Hilbert space. Consider a function $f:X\to\mathbb{R}$,
a point $x_{0}$ and a sequence $\{\epsilon_{i}\}$ of strictly decreasing
positive numbers converging to zero. This algorithm seeks to find
a point $x^{\prime}$ such that $f(x^{\prime})<0$.

\textbf{Step 0:} Set $i=0$.

\textbf{Step 1:} Find $s_{i}^{(j)}\in\partial f(x_{i})$ for $j=1,\dots,J_{i}$,
where $J_{i}$ is some finite number. Let $x_{i}^{(j)}$ be $x_{i}-\frac{\epsilon_{i}+f(x_{i})}{\|s_{i}^{(j)}\|^{2}}s_{i}^{(j)}$,
which is also the projection of $x_{i}$ onto the set ${\{x:f(x_{i})+\langle s_{i}^{(j)},x-x_{i}\rangle\leq-\epsilon_{i}\}}$.
Consider the polyhedron 
\[
P_{i}:=\big\{ x:f(x_{i})+\langle s_{i}^{(j)},x-x_{i}\rangle\leq-\epsilon_{i}\mbox{ for all }j\in\{1,\dots,J_{i}\}\big\}.
\]
The next iterate $x_{i+1}$ is obtained by projecting $x_{i}$ onto
$P_{i}$.

\textbf{Step 2:} Increase $i$ and go back to step 1 till convergence.
\end{algorithm}
Before proving Theorem \ref{thm:Fin-conv-NIP}, we recall a simple
principle that will be used in the proof there.
\begin{prop}
\label{prop:Proj-on-polyh}(Projections onto polyhedra) Let $F$ be
a polyhedron in a Hilbert space $X$ such that 
\[
F:=\bigcap_{i=1}^{k}\{x:\langle x,a_{i}\rangle\leq b_{i}\},
\]
where $a_{i}\in X$ and $b_{i}\in\mathbb{R}$ for $i\in\{1,\dots,k\}$.
Choose a point $x_{0}$ and let $x_{1}:=P_{F}(x_{0})$. Then for any
$y\in F$, we have $\langle x_{0}-x_{1},y-x_{1}\rangle\leq0$.
\end{prop}
We now prove that Algorithm \ref{alg:fin-conv-alg} can converge in
finitely many iterations to such a point $x^{\prime}$. Our proof
is an extension of the proof in \cite{Fukushima82}.
\begin{thm}
\label{thm:Fin-conv-NIP}(Finite convergence of Algorithm \ref{alg:fin-conv-alg})
Let $X$ be a Hilbert space. Consider a locally Lipschitz function
$f:X\to\mathbb{R}$. Let $\bar{x}$ be such that 
\begin{enumerate}
\item $f(\bar{x})=0$,
\item $0\notin\partial f(\bar{x})$, and
\item $f(\cdot)$ is approximately convex at $\bar{x}$. 
\end{enumerate}
Suppose also that the strictly decreasing sequence $\{\epsilon_{i}\}_{i}$
converges to zero at a sublinear rate (i.e., slower than any linearly
convergent sequence). There is a neighborhood $U$ of $\bar{x}$ and
a number $\bar{\epsilon}$ such that if $x_{0}\in U$ and $\epsilon_{0}<\bar{\epsilon}$,
then Algorithm \ref{alg:fin-conv-alg} converges in finitely many
iterations. (i.e., $f(x_{i})\leq0$ for some $i$.)\end{thm}
\begin{proof}
Seeking a contradiction, we assume $f(x_{i})>0$ for all $i$. Our
proof is broken up into several parts.

\textbf{Claim 1: }There is a neighborhood $U$ of $\bar{x}$ and $\bar{\epsilon}>0$
such that if $\tilde{x}_{1}\in U$ and $f(\tilde{x}_{1})>0$, then
for any $\tilde{\epsilon}\in(0,\bar{\epsilon}]$ and $s^{(j)}\in\partial f(\tilde{x}_{1})$,
where $j\in\{1,\dots,J\}$, the projection of $\tilde{x}_{1}$ onto
the polyhedron 
\begin{equation}
P:=\big\{ x:f(\tilde{x}_{1})+\langle s^{(j)},x-\tilde{x}_{1}\rangle\leq-\tilde{\epsilon}\mbox{ for all }j\in\{1,\dots,J\}\big\}\label{eq:poly-P}
\end{equation}
 lies in $U$.

By the Clarke directional derivative \eqref{eq:Clarke-dir-diff} of
$f(\cdot)$ at $\bar{x}$, since $0\notin\partial f(\bar{x})$, there
exists a direction $d$, where $\|d\|=1$, and $\mu>0$ such that
\[
\limsup_{t\to0}\frac{1}{t}[f(\bar{x}+td)-f(\bar{x})]<-\mu.
\]
 In particular, this implies that if $\bar{t}$ is small enough,
then $f(\bar{x}+td)<f(\bar{x})-\mu t=-\mu t$ for all $t\in[0,\bar{t}]$.
 Then by the approximate convexity of $f(\cdot)$ at $\bar{x}$,
for any $\epsilon_{ac}>0$, there is a neighborhood $U_{1}$ of $\bar{x}$
such that 
\[
f(y)\geq f(x)+\langle s^{(j)},y-x\rangle-\epsilon_{ac}\|y-x\|\mbox{ for all }x,y\in U_{1}\mbox{ and }j\in\{1,\dots,J\}.
\]

To simplify our notation, we let $S_{\epsilon}:=f^{-1}((-\infty,-\epsilon])$
just like in \cite{Fukushima82}. Recall $E(\cdot)$, the epigraphical
map of $f(\cdot)$ defined in Proposition \ref{prop:Metric-reg-epigph},
is metrically regular at $\bar{x}$. This means that by lowering $\bar{\epsilon}$
if necessary, there is a $\kappa\in[\bar{\kappa},\bar{\kappa}+1]$,
a $\delta>0$ and a neighborhood $U_{2}$ of $\bar{x}$ such that
if $x\in U_{2}$ and $\epsilon<\bar{\epsilon}$, then 
\begin{equation}
d(x,S_{\epsilon})=d\big(x,E^{-1}(-\epsilon)\big)\leq\kappa d\big(E(x),-\epsilon\big)=\kappa[f(x)+\epsilon].\label{eq:metric-reg-line}
\end{equation}
We impose the following requirements on $\epsilon_{ac}$, $\bar{t}$
and $\bar{\epsilon}$. 
\begin{enumerate}
\item [(R1)] Let $\epsilon_{ac}>0$ be small enough so that $2\epsilon_{ac}<\mu$
and $(\bar{\kappa}+1)\epsilon_{ac}<\frac{1}{3}$.
\item [(R2)] Let $\bar{t}>0$ be small enough so that 
\begin{eqnarray*}
 &  & \mathbb{B}(\bar{x}+\bar{t}d,2\bar{t})\subset U_{1}\cap U_{2},\\
 & \mbox{and} & L:=\sup_{s\in\partial f(\mathbb{B}(\bar{x}+\bar{t}d,2\bar{t}))}\|s\|\mbox{ is finite}.
\end{eqnarray*}

\item [(R3)] Let $\bar{\epsilon}>0$ be small enough so that $\bar{\epsilon}+2\epsilon_{ac}\bar{t}<\bar{t}\mu$.
\item [(R4)] Reduce $\bar{t}$ and $\bar{\epsilon}$ if necessary so that
(R2) and (R3) holds, and\\
 ${\mathbb{B}(\bar{x}+\bar{t}d,2\bar{t}+[\bar{\kappa}+1][3L\bar{t}+\bar{\epsilon}])\subset U_{1}}$.
\end{enumerate}
The finiteness of $L$ in (R2) is possible for some $\bar{t}>0$ by
making use of \cite[Proposition 2.1.2(a)]{Cla83} and the fact that
$f$ is locally Lipschitz at $\bar{x}$. We can now prove Claim 1
for $U=\mathbb{B}(\bar{x}+\bar{t}d,2\bar{t})$. We will only need
(R1)-(R3) for now, and the significance of (R4) will be explained
in Claim 2. Consider any $\tilde{x}_{1}\in\mathbb{B}(\bar{x}+\bar{t}d,2\bar{t})$
such that $f(\tilde{x}_{1})>0$. For any $s\in\partial f(\tilde{x}_{1})$,
we have 
\begin{eqnarray*}
 &  & f(\tilde{x}_{1})+\left\langle s,\frac{-[\tilde{\epsilon}+f(\tilde{x}_{1})]}{\|s\|^{2}}s\right\rangle \\
 & = & -\tilde{\epsilon}\\
 & > & -\bar{t}\mu+\epsilon_{ac}2\bar{t}\mbox{ (Using (R3) and the fact that }\tilde{\epsilon}\leq\bar{\epsilon}\mbox{.)}\\
 & > & f(\bar{x}+\bar{t}d)+\epsilon_{ac}2\bar{t}\\
 & \geq & f(\tilde{x}_{1})+\langle s,(\bar{x}+\bar{t}d)-\tilde{x}_{1}\rangle-\epsilon_{ac}\|[\bar{x}+\bar{t}d]-\tilde{x}_{1}\|+\epsilon_{ac}2\bar{t}\\
 & \geq & f(\tilde{x}_{1})+\langle s,(\bar{x}+\bar{t}d)-\tilde{x}_{1}\rangle\mbox{ (since }\tilde{x}_{1}\in\mathbb{B}(\bar{x}+\bar{t}d,2\bar{t})\mbox{)}.
\end{eqnarray*}
This implies that $\left\langle s^{(j)},\frac{-[\tilde{\epsilon}+f(\tilde{x}_{1})]}{\|s^{(j)}\|^{2}}s^{(j)}+\tilde{x}_{1}-[\bar{x}+\bar{t}d]\right\rangle \geq0$
for all $j\in\{1,\dots,J\}$. Let $\tilde{v}^{(j)}=\frac{[\tilde{\epsilon}+f(\tilde{x}_{1})]}{\|s^{(j)}\|^{2}}s^{(j)}$.
We have 
\begin{equation}
\big\langle\tilde{x}_{1}-[\tilde{x}_{1}-\tilde{v}^{(j)}],[\bar{x}+\bar{t}d]-[\tilde{x}_{1}-\tilde{v}^{(j)}]\big\rangle\leq0.\label{eq:angle-bigger-pi-2}
\end{equation}
In other words, the angle $\angle\tilde{x}_{1}[\tilde{x}_{1}-\tilde{v}^{(j)}][\bar{x}+\bar{t}d]\geq\pi/2$.

The polyhedron $P$ in \eqref{eq:poly-P} can also be written as 
\[
P=\big\{ x:\langle x-[\tilde{x}_{1}-\tilde{v}^{(j)}],s^{(j)}\rangle\leq0\mbox{ for all }j\in\{1,\dots,J\}\big\}.
\]
In view of \eqref{eq:angle-bigger-pi-2} and the above discussion,
the point $\bar{x}+\bar{t}d$ lies in $P$. The projection of $\tilde{x}_{1}$
onto $P$, say $\tilde{x}_{2}$, creates a hyperplane that separates
$\tilde{x}_{1}$ and $\bar{x}+\bar{t}d$. In other words, $\angle\tilde{x}_{1}\tilde{x}_{2}[\bar{x}+\bar{t}d]\geq\pi/2$.
This in turn implies that we have $\|\tilde{x}_{2}-[\bar{x}+\bar{t}d\|\leq\|\tilde{x}_{1}-[\bar{x}+\bar{t}d]\|$.
In other words, $\tilde{x}_{1}\in\mathbb{B}(\bar{x}+\bar{t}d,2\bar{t})$
implies $\tilde{x}_{2}\in\mathbb{B}(\bar{x}+\bar{t}d,2\bar{t})$.
This ends the proof of Claim 1 with $U=\mathbb{B}(\bar{x}+\bar{t}d,2\bar{t})$.

It is easy to see that this implies that if $x_{0}\in\mathbb{B}(\bar{x}+\bar{t}d,2\bar{t})$,
then the iterates $x_{i}$ generated by Algorithm \ref{alg:fin-conv-alg}
lie in $\mathbb{B}(\bar{x}+\bar{t}d,2\bar{t})$ as well, provided
the starting $\epsilon_{0}$ is smaller than $\bar{\epsilon}$.

\textbf{Claim 2}: Let $p_{i}:=P_{S_{\epsilon_{i}}}(x_{i})$, the projection
of $x_{i}$ onto $S_{\epsilon_{i}}$. If $x_{i}\in\mathbb{B}(\bar{x}+\bar{t}d,2\bar{t})$,
then $p_{i}$ lies in $U_{1}$.

From \eqref{eq:metric-reg-line}, we have 
\begin{equation}
\|p_{i}-x_{i}\|=d(x_{i},S_{\epsilon_{i}})\leq\kappa[f(x_{i})+\epsilon_{i}]\leq[\bar{\kappa}+1][f(x_{i})+\epsilon_{i}].\label{eq:metric-reg-bdd2}
\end{equation}
Since $x_{i}\in\mathbb{B}(\bar{x}+\bar{t}d,2\bar{t})$, we have $\|x_{i}-\bar{x}\|\leq3\bar{t}$.
It is well known that the constant $L$ in (R2) is also an upper bound
on the Lipschitz constant of $f$ in $\mathbb{B}(\bar{x}+\bar{t}d,2\bar{t})$
(for example, through the Mean Value Theorem in \cite[Theorem 2.3.7]{Cla83}
or \cite{Leb75}), so $f(x_{i})$ is bounded from above by $3L\bar{t}$.
Hence 
\[
\|p_{i}-(\bar{x}+\bar{t}d)\|\leq\|x_{i}-(\bar{x}+\bar{t}d)\|+\|p_{i}-x_{i}\|\leq2\bar{t}+[\bar{\kappa}+1][3L\bar{t}+\bar{\epsilon}].
\]
By (R4), we can see that $p_{i}\in U_{1}$ as needed. This ends the
proof of Claim 2.

\textbf{Claim 3:} The sequence $\{d(x_{i},S_{\epsilon_{i}})\}_{i}$
converges at least linearly to $0$.

From the continuity of $f(\cdot)$, it is clear that $f(p_{i})=-\epsilon_{i}$.
For the choice $s_{i}^{(j)}\in\partial f(x_{i})$, we recall that
$x_{i},p_{i}\in U_{1}$, and get 
\begin{eqnarray*}
f(p_{i}) & \geq & f(x_{i})+\langle s_{i}^{(j)},p_{i}-x_{i}\rangle-\epsilon_{ac}\|p_{i}-x_{i}\|\\
\langle s_{i}^{(j)},p_{i}-x_{i}\rangle & \leq & f(p_{i})-f(x_{i})+\epsilon_{ac}\|p_{i}-x_{i}\|\\
 & = & -\epsilon_{i}-f(x_{i})+\epsilon_{ac}\|p_{i}-x_{i}\|.
\end{eqnarray*}
Recall $x_{i}^{(j)}=x_{i}-\frac{[\epsilon_{i}+f(x_{i})]}{\|s\|^{2}}s_{i}^{(j)}$.
Let 
\begin{equation}
\tilde{x}_{i}^{(j)}:=\frac{2}{3}x_{i}^{(j)}+\frac{1}{3}x_{i}.\label{eq:scale-x-i-j}
\end{equation}
It is easy to check that $\langle s_{i}^{(j)},x_{i}^{(j)}-x_{i}\rangle=-\epsilon_{i}-f(x_{i})$
and $\langle s_{i}^{(j)},\tilde{x}_{i}^{(j)}-x_{i}\rangle=-\frac{2}{3}[\epsilon_{i}+f(x_{i})]$.

From \eqref{eq:metric-reg-bdd2} and the preceeding discussion, we
have 
\begin{eqnarray}
\frac{\langle s_{i}^{(j)},p_{i}-x_{i}\rangle}{\langle s_{i}^{(j)},\tilde{x}_{i}^{(j)}-x_{i}\rangle} & \geq & \frac{[\epsilon_{i}+f(x_{i})]-\epsilon_{ac}\|x_{i}-p_{i}\|}{\frac{2}{3}[\epsilon_{i}+f(x_{i})]}\label{eq:explain-0}\\
 & \geq & \frac{\|x_{i}-p_{i}\|/[\bar{\kappa}+1]-\epsilon_{ac}\|x_{i}-p_{i}\|}{\frac{2}{3}\|x_{i}-p_{i}\|/[\bar{\kappa}+1]}\nonumber \\
 & = & \frac{3}{2}\big[1-[\bar{\kappa}+1]\epsilon_{ac}\big].\nonumber 
\end{eqnarray}
In view of $[\bar{\kappa}+1]\epsilon_{ac}<1/3$ in (R1), the ratio
$\frac{3}{2}[1-[\bar{\kappa}+1]\epsilon_{ac}]$ is greater than $1$.
Since $\tilde{x}_{i}^{(j)}-x_{i}$ is in the direction of $-s_{i}^{(j)}$,
the angle $\angle p_{i}\tilde{x}_{i}^{(j)}x_{i}$ is greater than
$\pi/2$. (See Figure \ref{fig:explain}.) In other words, the point
$p_{i}$ is in the polyhedron 
\[
\tilde{P}_{i}:=\big\{ x:\langle x-\tilde{x}_{i}^{(j)},s_{i}^{(j)}\rangle\leq0\mbox{ for all }j\in\{1,\dots,J_{i}\}\big\}.
\]

\begin{figure}[!h]
\includegraphics[scale=0.3]{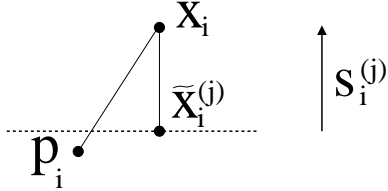}\caption{\label{fig:explain}This figure explains the setup in \eqref{eq:explain-0}
and why $\angle p_{i}\tilde{x}_{i}^{(j)}x_{i}>\pi/2$. }

\end{figure}
Let the projection of $x_{i}$ onto $\tilde{P}_{i}$ be $\tilde{x}_{i+1}$.
It is clear to see that $\tilde{P}_{i}$ is the polyhedron created
by scaling $P_{i}$ about $x_{i}$ with a factor of $2/3$ by \eqref{eq:scale-x-i-j}.
Thus $\tilde{x}_{i+1}=\frac{2}{3}x_{i+1}+\frac{1}{3}x_{i}$. We can
also infer (using the principle in Proposition \ref{prop:Proj-on-polyh})
that 
\begin{equation}
\angle x_{i}\tilde{x}_{i+1}p_{i}\geq\pi/2.\label{eq:angle-bigger-than-pi-2}
\end{equation}
Let $\tilde{p}_{i}$ be the projection of $p_{i}$ onto the line connecting
$x_{i+1}$ and $x_{i}$. We can infer from \eqref{eq:angle-bigger-than-pi-2}
that $\tilde{x}_{i+1}$ must lie between $\tilde{p}_{i}$ and $x_{i}$.
Thus 
\begin{eqnarray}
\|x_{i+1}-p_{i}\|^{2} & = & \|x_{i}-p_{i}\|^{2}-\|\tilde{p}_{i}-x_{i}\|^{2}+\|\tilde{p}_{i}-x_{i+1}\|^{2}\label{eq:1-3-ineq}\\
 & \leq & \|x_{i}-p_{i}\|^{2}-\|\tilde{x}_{i+1}-x_{i}\|^{2}+\|\tilde{x}_{i+1}-x_{i+1}\|^{2}\nonumber \\
 & = & \|x_{i}-p_{i}\|^{2}-\frac{4}{9}\|x_{i+1}-x_{i}\|^{2}+\frac{1}{9}\|x_{i+1}-x_{i}\|^{2}\nonumber \\
 & = & \|x_{i}-p_{i}\|^{2}-\frac{1}{3}\|x_{i+1}-x_{i}\|^{2}.\nonumber 
\end{eqnarray}
We now bound the distance $\|x_{i+1}-x_{i}\|$.  From \eqref{eq:metric-reg-line},
we have 
\[
\|p_{i}-x_{i}\|\leq\kappa[f(x_{i})+\epsilon_{i}]=\kappa\left\Vert \frac{[f(x_{i})+\epsilon_{i}]}{\|s_{i}^{(1)}\|^{2}}s_{i}^{(1)}\right\Vert \|s_{i}^{(1)}\|\leq\kappa\|x_{i}-x_{i}^{(1)}\|\sup_{s^{\prime}\in\partial f(x_{i})}\|s^{\prime}\|.
\]
We let $r:=1/[\kappa\sup_{s^{\prime}\in\partial f(\mathbb{B}(\bar{x}+\bar{t}d,2\bar{t}))}\|s^{\prime}\|]$.
Then we have 
\begin{equation}
\|x_{i}-x_{i}^{(1)}\|\geq r\|p_{i}-x_{i}\|.\label{eq:x-1-ineq}
\end{equation}
Since $x_{i}^{(1)}$ is the projection of $x_{i}$ onto the halfspace
\[
H_{i}^{(1)}:=\{x:\langle x-x_{i}^{(1)},s_{i}^{(1)}\rangle\leq0\},
\]
$x_{i+1}$ is the projection of $x_{i}$ onto $P_{i}$, and $P_{i}\subset H_{i}^{(1)}$,
we must have $\|x_{i+1}-x_{i}\|\geq\|x_{i}-x_{i}^{(1)}\|$. Combining
with \eqref{eq:1-3-ineq} and \eqref{eq:x-1-ineq}, we have 
\begin{eqnarray*}
\|x_{i+1}-p_{i}\|^{2} & \leq & \|x_{i}-p_{i}\|^{2}-\frac{1}{3}\|x_{i+1}-x_{i}\|^{2}\\
 & \leq & \|x_{i}-p_{i}\|^{2}-\frac{1}{3}r^{2}\|x_{i}-p_{i}\|^{2}\quad=\quad\left[1-\frac{1}{3}r^{2}\right]\|x_{i}-p_{i}\|^{2}.
\end{eqnarray*}
Hence 

\[
d(x_{i+1},S_{\epsilon_{i+1}})\leq d(x_{i+1},S_{\epsilon_{i}})\leq\sqrt{1-\frac{1}{3}r^{2}}\|x_{i}-p_{i}\|=\sqrt{1-\frac{1}{3}r^{2}}d(x_{i},S_{\epsilon_{i}}),
\]
which gives at least a linear rate of decrease of $\{d(x_{i},S_{\epsilon_{i}})\}_{i}$.
This ends the proof of Claim 3.

\textbf{Claim 4}: The sequence $\{d(x_{i},S_{\epsilon_{i}})\}_{i}$
converges to $0$ at a sublinear rate, contradicting Claim 3.

Recall $p_{i}=P_{S_{\epsilon_{i}}}(x_{i})$. We have 
\begin{eqnarray*}
-\epsilon_{i} & = & f(p_{i})\\
 & \geq & f(x_{i})+\langle s_{i},p_{i}-x_{i}\rangle-\epsilon_{ac}\|p_{i}-x_{i}\|\\
 & \geq & \langle s_{i},p_{i}-x_{i}\rangle-\epsilon_{ac}\|p_{i}-x_{i}\|\\
 & \geq & -[\|s_{i}\|+\epsilon_{ac}]\|p_{i}-x_{i}\|,
\end{eqnarray*}
which gives $\|x_{i}-p_{i}\|\geq\frac{\epsilon_{i}}{\|s_{i}\|+\epsilon_{ac}}$,
so 
\[
d(x_{i},S_{\epsilon_{i}})=\|x_{i}-p_{i}\|\geq\frac{\epsilon_{i}}{\|s_{i}\|+\epsilon_{ac}}\geq\frac{\epsilon_{i}}{L+\epsilon_{ac}}.
\]
This implies that the sequence $\{d(x_{i},S_{\epsilon_{i}})\}$ converges
at a sublinear rate, which contradicts Claim 3. This ends the proof
of Claim 4.

Thus, the sequence $\{x_{i}\}$ has to terminate finitely.\end{proof}
\begin{acknowledgement}
I thank Shawn Wang and Heinz Bauschke for conversations about the
material in this paper.
\end{acknowledgement}
\bibliographystyle{amsalpha}
\bibliography{../refs}

\end{document}